\documentclass[11pt]{article}

\usepackage{amsmath} 
\usepackage{mathrsfs}
\usepackage{amssymb} 
\usepackage{theorem}
\usepackage{rotating}
\usepackage{stmaryrd}
\usepackage[all,cmtip]{xy}
\usepackage{pict2e}

\theoremstyle{change} 
\newtheorem{theorem}{Theorem}[section] 
\newtheorem{lemma}[theorem]{Lemma} 
\newtheorem{proposition}[theorem]{Proposition}

\theorembodyfont{\rmfamily} 
\newtheorem{remark}[theorem]{Remark}

\newtheorem{definition}[theorem]{Definition}
\newtheorem{notation}[theorem]{Notation}
\newtheorem{nothing}[theorem]{} 

\newenvironment{proof}{\noindent{\bf Proof}\ }{\qed\bigskip}

\renewcommand{\le}{\leqslant}
\renewcommand{\leq}{\leqslant}

\renewcommand{\marginpar}[1]{}

\newcommand{\BIGOP}[1]
  {\mathop{\mathchoice
  {\raise-0.22em\hbox{\huge $#1$}}
  {\raise-0.05em\hbox{\Large $#1$}}{\hbox{\large $#1$}}{#1}}}

\newcommand{\catC}{\catfont{C}}

\newcommand{\catfont}{\mathsf}
\newcommand{\CC}{\mathbb{C}}
\newcommand{\disj}{\buildrel.\over\cup}
\newcommand{\bigdisj}{\mathop{\buildrel.\over\bigcup}}

\newcommand{\End}{\mathrm{End}}

\newcommand{\Hom}{\mathrm{Hom}}
\newcommand{\Hd}{\mathrm{Hd}}

\newcommand{\Mor}{\mathrm{Mor}}

\newcommand{\liso}{\buildrel\sim\over\longrightarrow}

\newcommand{\Ob}{\mathrm{Ob}}

\newcommand{\qed}{\nobreak\hfill
                   \vbox{\hrule\hbox{\vrule\hbox to 5pt
                   {\vbox to 8pt{\vfil}\hfil}\vrule}\hrule}}

\newcommand{\Rad}{\mathrm{Rad}}

\newcommand{\scrJ}{\mathscr{J}}

\title{Twisted split category algebras as quasi-hereditary algebras
\footnote{{\bf MR Subject Classification:}  16G10, 20M17.
{\bf Keywords:}  twisted category algebra, regular semigroup, quasi-hereditary algebra}}

\author{\small Robert Boltje\\
  \small Department of Mathematics\\
  \small University of California\\
  \small Santa Cruz, CA 95064\\
  \small U.S.A.\\
  \small boltje@ucsc.edu
  \and
  \small Susanne Danz\\
  \small Department of Mathematics\\ 
  \small University of Kaiserslautern\\
  \small P.O. Box 3049\\
  \small 65653 Kaiserslautern\\
  \small Germany\\
  \small danz@mathematik.uni-kl.de}
\date{August, 14, 2012}

\begin{document}
\sloppy


\maketitle


\begin{abstract}
We show that if $\catC$ is a finite split category, $k$ is a field of characteristic $0$ and $\alpha$ is a $2$-cocycle of $\catC$ with values in $k^\times$ then the twisted category algebra $k_\alpha\catC$ is quasi-hereditary.
\end{abstract}


\section{Introduction}\label{sec intro}
Throughout this paper we assume that $\catC$ is a finite category, that is,
the objects of $\catC$ form a finite set, and for every $X,Y\in\Ob(\catC)$,
the morphism set $\Hom_\catC(X,Y)$ is finite.
The category $\catC$ is called {\it split} if, for each morphism $s\in\Hom_\catC(X,Y)$, there
is a (not necessarily unique)
morphism $t\in\Hom_\catC(Y,X)$ such that $s\circ t\circ s=s$. 
Note that $u:=t\circ s\circ t$ then also satisfies $s\circ u\circ s=s$,
and also $u\circ s\circ u=u$.
In the special case where $\catC$ has only one object this leads to
the notion of a {\it regular monoid}, see \cite{Gr1}.

Let $k$ be a field, and let
$\alpha$ be a $2$-cocycle of $\catC$ with values in $k^\times$. That is,
for every pair $s,t\in\Mor(\catC)$ such that $t\circ s$
exists, one has an element $\alpha(t,s)\in k^\times$ such that the following holds:
for any $s,t,u\in\Mor(\catC)$ such that $t\circ s$ and $u\circ t$ exist,
one has $\alpha(u\circ t,s)\alpha(u,t)=\alpha(u,t\circ s)\alpha(t,s)$. 
We will study the twisted category algebra $k_\alpha\catC$, that is,
the $k$-vector space with basis $\Mor(\catC)$ and multiplication
$$t\cdot s:=\begin{cases}
           \alpha(t,s)\cdot t\circ s & \text{ if } t\circ s \text{ exists,}\\
           0                         & \text{ otherwise.} 
\end{cases}$$

\smallskip
The aim of this paper, see Theorem~\ref{thm qh}, is to show that if $\catC$ is a finite split category and if $k$ has characteristic $0$ then $k_\alpha \catC$ is a quasi-hereditary algebra. This generalizes a result of Putcha, see~\cite{Putcha}, who proved that regular monoid algebras are quasi-hereditary over $k=\CC$. In Theorem~\ref{thm std modules} we identify the standard modules, generalizing Putcha's results in \cite{Putcha}.

\smallskip
Our main motivation for studying the quasi-hereditary structure of twisted category algebras comes from the theory of double Burnside rings and biset functors: by
a result of Webb, see \cite{Webb}, the category of biset functors over a field of characteristic $0$ is a highest weight category. In \cite[Example 5.15(b)]{BDII} we introduced an algebra $A$ with the property that the category of biset functors (on a finite set of groups) over a field of characteristic $0$ is equivalent to the category of $eAe$-modules, where $e$ is an idempotent of $A$.
Thus, by Webb's result, $eAe$ is a quasi-hereditary algebra. It is natural to ask whether also $A$ is quasi-hereditary. In \cite{BDII} it was also shown that $A$ is a twisted category algebra for a finite split category. Thus 
Theorem~\ref{thm qh} of the present paper, in particular, implies that the algebra $A$ in \cite{BDII} is indeed quasi-hereditary.

\smallskip
We further remark that Theorems~\ref{thm qh} and \ref{thm std modules} should be of independent interest, since, by work of Wilcox~\cite{Wil}, they also cover various prominent classes of cellular algebras (for suitable parameters) such as Brauer algebras, cyclotomic Brauer algebras, Temperley--Lieb algebras, and partition algebras, so that the main result of this paper gives a unified proof for the known fact that these algebras are quasi-hereditary.
Proofs of the quasi-heredity of the aforementioned diagram algebras can, for instance,
be found in \cite{GL, M, RX, RY, West, X},
and also in work of
K\"onig--Xi \cite{KX}, who established necessary and sufficient
criteria for a cellular algebra
to be quasi-hereditary.
For a more detailed discussion of the history of proofs that Brauer algebras, Temperley--Lieb algebras, and partition algebras are quasi-hereditary over coefficient fields of characteristic $0$, we refer to \cite{LS2}.

\smallskip
We recently learnt that Linckelmann and Stolorz, see \cite[Theorem~1.1]{LS2}, independently proved that, under certain conditions on the category, finite twisted category algebras are quasi-hereditary in characteristic $0$. These conditions on the category are even weaker than being split, and therefore the results in \cite{LS2} imply Theorem~\ref{thm qh}. However, the two approaches are slightly different; for instance, we explicitly determine the radical of the twisted category algebra as part of our proof. In addition, we construct the standard modules of $k_\alpha\catC$. 
As has been pointed out by the referee, in the case where $k_\alpha\catC$
is isomorphic to a Brauer algebra and $\mathrm{char}(k)=0$, standard
modules have also been investigated by Cox--De Visscher--Martin
in \cite{CdVM}; standard modules for cyclotomic Brauer algebras over fields 
of characteristic 0 have recently been studied by Bowman--Cox--De Visscher in
\cite{BCdV}.

\medskip
{\bf Acknowledgement.} The authors' research has been supported through a Research in Pairs Grant of the Mathematisches Forschungsinstitut Oberwolfach in February 2012, and through DFG grant DA 1115/3-1.


\section{Notation and quoted results}\label{sec notation}

Throughout this section we assume that $\catC$ is a finite split category.
We begin by collecting some known facts concerning split
categories that will be used repeatedly in this paper. 
For details and proofs of the results quoted here we refer the
reader to \cite{LS} and \cite{GMS}.

In what follows, given subsets $S$ and  $T$ of $\Mor(\catC)$, we set
$S\circ T:=\{s\circ t\mid s\in S,\, t\in T\text{ such that $s\circ t$ exists}\}$. In the case
where $S=\{s\}$ or $T=\{t\}$, we abbreviate $S\circ T$ by
$s\circ T$ or $S\circ t$, respectively. Note that $S\circ T$ may be 
empty, even if neither $S$ nor $T$ is empty. 

One calls $S$ a {\it left ideal} (respectively, {\it right ideal})
of $\catC$ if $\Mor(\catC)\circ S\subseteq S$ (respectively,
$S\circ\Mor(\catC)\subseteq S$). Note that this is
equivalent to $\Mor(\catC)\circ S= S$ (respectively,
$S\circ\Mor(\catC)= S$), since every object has an identity morphism.
Analogously, one calls $S$ a {\it (two-sided) ideal} of $\catC$
if $\Mor(\catC)\circ S\circ \Mor(\catC) \subseteq S$, or equivalently,
if $\Mor(\catC)\circ S\circ \Mor(\catC)= S$.

\begin{nothing}\label{noth J-classes}
{\bf Idempotents and $\scrJ$-classes.}\, (a)\, For 
morphisms $s,t\in\Mor(\catC)$ one defines
$$s\, \scrJ\, t :\Leftrightarrow \Mor(\catC)\circ s\circ\Mor(\catC)=\Mor(\catC)\circ t\circ\Mor(\catC)\, .$$
This yields an equivalence relation $\scrJ$ on the set $\Mor(\catC)$, and the
corresponding equivalence classes are called the {\it $\scrJ$-classes}
of $\catC$. We will denote the $\scrJ$-class of a morphism $s\in\Mor(\catC)$
by $\scrJ(s)$.

\smallskip
(b)\, Let $I$ and $J$ be $\scrJ$-classes of $\catC$. One sets
$$J\leq_\scrJ I:\Leftrightarrow \Mor(\catC)\circ J\circ\Mor(\catC)\subseteq \Mor(\catC)\circ I\circ \Mor(\catC)\, .$$
Note that this is also equivalent to 
$\Mor(\catC)\circ s\circ \Mor(\catC)\subseteq \Mor(\catC)\circ u\circ\Mor(\catC)$, 
where $s$ and $u$ are any representatives of $J$ and
$I$, respectively. Note further that this defines a poset structure
on the set of $\scrJ$-classes of $\catC$.

\smallskip
(c)\, An {\it idempotent} of $\catC$ is an endomorphism
$e\in\End_\catC(X)$ such that $e\circ e=e$, where $X$ is an object of
$\catC$. In this case we call $e$ an {\it idempotent on $X$}.

We say that idempotents $e$ on $X$ and $f$ on $Y$ are {\it equivalent}
if there exist some $s\in e\circ\Hom_\catC(Y,X)\circ f$ and
some $t\in f\circ\Hom_\catC(X,Y)\circ e$ such that $e=s\circ t$ and $f=t\circ s$.
In this case we write $e\sim f$. It is
straightforward to show that this defines an equivalence relation
on the set of idempotents of $\catC$; we will denote the 
equivalence class of an idempotent $e$ by $[e]$.

\smallskip

(d)\, 
The next lemma shows that every $\scrJ$-class of $\catC$ contains an
idempotent. Furthermore, idempotents
$e$ and $f$ of $\catC$ are equivalent if and only if $\scrJ(e)=\scrJ(f)$;
a proof of this can be found in \cite[Lemma 2.1]{LS}.
Thus there is a bijection between the equivalence classes of idempotents of $\catC$ and the
$\scrJ$-classes of $\catC$.
\end{nothing}

\begin{lemma}\label{lemma id equiv}
Let $s\in\Mor(\catC)$, and let $t,u\in\Mor(\catC)$ be such that
$s\circ t\circ s=s=s\circ u\circ s$. 
Then $s\circ t$ and $u\circ s$ are idempotents in 
$\catC$.
Moreover, 
$$\scrJ(s\circ t)=\scrJ(s)=\scrJ(u\circ s)\,;$$
in particular, $s\circ t\sim u\circ s$.
\end{lemma}

\begin{proof}
Clearly, $s\circ t$ and $u\circ s$ are idempotents in 
$\catC$ contained in the $\scrJ$-class $\scrJ(s)$. Thus, as already
mentioned, \cite[Lemma 2.1]{LS} implies $s\circ t\sim u\circ s$.
\end{proof}

Suppose now that $k$ is a field, and let
$\alpha$ be a $2$-cocycle of $\catC$ with values in $k^\times$. The aim of the next
section is to prove that, under suitable additional assumptions on $k$,
the $k$-algebra
$k_\alpha\catC$ is {\it quasi-hereditary}.
To this end, we summarize and establish here some important facts concerning the
algebra $k_\alpha\catC$ and, in particular, its simple modules and its 
Jacobson radical. For ease of notation, we 
will henceforth denote the twisted category algebra
$k_\alpha\catC$ by $A$.

\begin{nothing}\label{noth simples}
{\bf Idempotents of $\catC$ and simple $A$-modules.}\, 
The isomorphism classes of simple $A$-modules have been 
parametrized by Linckelmann--Stolorz \cite{LS}, generalizing previous
work of Ganyushkin--Mazorchuk--Steinberg \cite{GMS} concerning
semigroup algebras. 

\smallskip

(a)\, Given
an idempotent $e$ on $X\in\Ob(\catC)$, the group of invertible elements
of the monoid $e\circ\End_\catC(X)\circ e$ is denoted by $\Gamma_e$, and is
called a {\it maximal subgroup} of $\catC$. Moreover,
we set $J_e:=e\circ\End_\catC(X)\circ e\smallsetminus \Gamma_e$.
Restricting the $2$-cocycle $\alpha$ to $\Gamma_e$, one can
view the twisted group algebra $k_\alpha\Gamma_e$ as (non-unitary)
subalgebra of $A$.

Note also that the element
$$e':=\alpha(e,e)^{-1} e$$
is an idempotent in the algebra $A$, and that 
$eAe=e'Ae'=k_\alpha(e\circ\End_\catC(X)\circ e)$. Furthermore, there is a
$k$-vector space decomposition
\begin{equation}\label{eqn Gamma-J-decomposition}
e'Ae'=k_\alpha\Gamma_e\oplus kJ_e\,,
\end{equation}
$kJ_e$ is a two-sided ideal, and $k_\alpha\Gamma_e$ is a unitary 
subalgebra of $e'Ae'$.

\smallskip
(b)\, Suppose that $e$ is an idempotent of $\catC$,
and let again $e'$ denote the corresponding idempotent in $A$.
Whenever $W$ is a $k_\alpha\Gamma_e$-module, we obtain an $A$-module
$Ae'\otimes_{e'Ae'}\tilde{W}$, where $\tilde{W}$ is the inflation
of $W$ from $k_\alpha\Gamma_e$ to $e'Ae'$ with respect to the 
decomposition (\ref{eqn Gamma-J-decomposition}).
In the case where $W$ is a simple $k_\alpha\Gamma_e$-module,
the $A$-module $Ae'\otimes_{e'Ae'}\tilde{W}$ has a unique
simple quotient module; see \cite[Section 6.2]{Gr}.

\smallskip
(c)\, Let $e\in\End_{\catC}(X)$ and $f\in\End_{\catC}(Y)$ be equivalent idempotents of $\catC$ and let $s\in e\circ\Hom_{\catC}(Y,X)\circ f$ and $t\in f\circ\Hom_{\catC}(X,Y)\circ e$ be such that $e=s\circ t$ and $f=t\circ s$. Then the map $a\mapsto t\cdot a\cdot s$ defines a $k$-linear isomorphism between $e'A e'$ and $f'Af'$, which takes $kJ_e$ to $kJ_f$. Similarly, one obtains an isomorphism of left $A$-modules between $Ae'$ and $Af'$. Altogether one obtains an isomorphism of left $A$-modules
\begin{equation}\label{eqn indep of e}
  Ae'\otimes_{e'Ae'} (e'Ae'/kJ_e) \liso Af'\otimes_{f'Af'} (f'Af'/kJ_f)\,.
\end{equation}
\end{nothing}

\begin{notation}\label{not simples}
From now on, we denote by $e_1,\ldots,e_n$ representatives 
of the equivalence classes of idempotents of $\catC$, and for $i=1,\ldots,n$
we fix representatives $T_{i1},\ldots,T_{il_i}$ of the isomorphism classes
of simple $k_\alpha\Gamma_{e_i}$-modules. Moreover, for
$i=1,\ldots,n$ and $j=1,\ldots,l_i$, we denote the inflation of the 
$k_\alpha\Gamma_{e_i}$-module $T_{ij}$ to $e_i'Ae_i'$ by $\tilde{T}_{ij}$, and the
simple head of the $A$-module $A e_i'\otimes_{e_i'Ae_i'}\tilde{T}_{ij}$
by $D_{ij}$. With this, the following holds:
\end{notation}

\begin{theorem}[\cite{LS}, Theorem 1.2]\label{thm simples}
The modules $D_{ij}$ ($i=1,\ldots,n$, 
$j=1,\ldots,l_i$) form a set of representatives of the isomorphism
classes of simple $A$-modules.
\end{theorem}

Denoting the Jacobson radical of $A$ by $\mathbf{J}(A)$, Theorem~\ref{thm simples}
now leads to the following description:

\begin{proposition}\label{prop radical}
With the notation as in \ref{not simples}
one has
\begin{equation}\label{eqn radical}
\mathbf{J}(A)=\{u\in A\mid \forall\, i=1,\ldots,n: e_i'AuAe_i'\subseteq kJ_{e_i}+\mathbf{J}(k_\alpha \Gamma_{e_i})\}\,.
\end{equation}
In particular, if in addition $|\Gamma_{e_i}|\in k^\times$, for all 
$i=1,\ldots,n$, then
\begin{equation}\label{eqn radical2}
\mathbf{J}(A)=\{u\in A\mid \forall\, i=1,\ldots,n: e_i'AuAe_i'\subseteq kJ_{e_i}\}\,.
\end{equation}
\end{proposition}

\begin{proof}
It suffices to prove that the set on the right-hand side of (\ref{eqn radical})
is the common annihilator of the simple $A$-modules $D_{ij}$ 
($i=1,\ldots,n$, $j=1,\ldots,l_i$). So let $u\in A$, let 
$i\in\{1,\ldots,n\}$, and let $j\in\{1,\ldots,l_i\}$. By \cite[Lemma~5.6]{BDII}, we know that
$uD_{ij}=0$ if and only if $e_i'AuAe_i'\subseteq \mathrm{Ann}_{e_i'Ae_i'}(\tilde{T}_{ij})=kJ_{e_i}+\mathrm{Ann}_{k_\alpha \Gamma_{e_i}}(T_{ij}).$ Therefore,
\begin{align*}
u\in \mathbf{J}(A) &\Leftrightarrow \forall\, i=1,\ldots,n,\, j=1,\ldots, l_i\colon uD_{ij}=0\\
&\Leftrightarrow \forall\, i=1,\ldots,n\colon e_i'AuAe_i'\subseteq \bigcap_{j=1}^{l_i}(kJ_{e_i}+\mathrm{Ann}_{k_\alpha\Gamma_{e_i}}(T_{ij}))\\
&\quad \quad \quad\quad \quad \quad \quad \quad \quad\quad \quad \;\;=kJ_{e_i}+\bigcap_{j=1}^{l_i}\mathrm{Ann}_{k_\alpha\Gamma_{e_i}}(T_{ij})
=kJ_{e_i}+\mathbf{J}(k_\alpha\Gamma_{e_i})\,,
\end{align*}
proving (\ref{eqn radical}).
If, moreover $|\Gamma_{e_i}|\in k^\times$, for 
$i=1,\ldots,n$ then, by \cite[Exercise 28.4]{CR}, the twisted group algebras 
$k_\alpha\Gamma_{e_i}$
($i=1,\ldots,n$) are semisimple, and we derive
equation~(\ref{eqn radical2}).
\end{proof}


\section{A heredity chain for $k_\alpha\catC$}\label{sec qh}

In this section we will prove the main result, Theorem~\ref{thm qh}. We start by recalling the definition of a quasi-hereditary algebra.

\begin{definition}[Cline--Parshall--Scott \cite{CPS}, Section~3]\label{defi qh}
Let $k$ be any field. A finite-dimensional $k$-algebra $A$ is called
{\it quasi-hereditary} if there exists a chain
\begin{equation}\label{eqn chain}
\{0\}=J_0\subset J_1\subset\cdots\subset J_{n-1}\subset J_n=A
\end{equation}
of
two-sided ideals in $A$ such that, for every $i=1,\ldots,n$, 
when denoting by $\bar{\cdot}\colon A\to A/J_{i-1}=\bar{A}$
the canonical epimorphism, the following
conditions are satisfied:

\smallskip

(i)\, there is an idempotent $\bar{e_i}\in\bar{A}$ with $\bar{J}_i=\bar{A}\bar{e_i}\bar{A}$;

(ii)\, $\bar{J}_i\cdot \mathbf{J}(\bar{A})\cdot \bar{J}_i=\{0\}$;

(iii)\, $\bar{J}_i$ is a projective right $\bar{A}$-module.

\smallskip
\noindent
In this case one calls the chain (\ref{eqn chain}) a {\it heredity chain}
for $A$. Note that (iii) can be replaced by

(iii')\, $\bar{J}_i$ is a projective left $\bar{A}$-module,

\noindent
by \cite[Statement~7]{DR}.
\end{definition}

For the remainder of this section assume again that $\catC$ is a finite split category. For ease of notation we denote from now on the morphism set $\Mor(\catC)$ by $S$.
Recall from Section~\ref{sec notation} that we also have the notions of 
left/right/two-sided ideals  
of the category $\catC$. So we will now use the term `ideal' both in
the context of categories and algebras.

We will next define a chain of two-sided ideals of $\catC$ that will then
give rise to
a heredity chain for the twisted category algebra in Theorem~\ref{thm qh}.

\begin{definition}\label{defi ideals S}
Let $e_1,\ldots,e_n$ be representatives of the equivalence
classes of idempotents of $\catC$, ordered such that
$$\scrJ(e_i)\leq_\scrJ \scrJ(e_j)\quad \text{ implies }\quad i\leq j\,,$$
in which case we also write $i\leq_\scrJ j$. Moreover, for $i=1,\ldots,n$, we
define
$$S_i:=\scrJ(e_i),\quad S_{\leq_\scrJ i}:=\bigdisj_{j\leq_\scrJ i} S_j,\quad S_{\leq i}:=\bigdisj_{j\leq i} S_j\,.$$
Then $S_n=S$, and for convenience we also set $S_0:=S_{\leq 0}:=\emptyset\subseteq S$.
Note that, by \cite[Lemma~2.6]{LS}, one has 
\begin{equation}\label{eqn Gamma J}
  \Gamma_e = (e\circ S\circ e) \cap S_i \quad\text{and}\quad J_e = (e\circ S\circ e)\cap S_{\le i-1}\,,
\end{equation}
for all $i\in\{1,\ldots,n\}$ and any idempotent $e\in S_i$.
\end{definition}

\begin{proposition}\label{prop ideals S}
With the notation as in Definition~\ref{defi ideals S}, for
$i=1,\ldots,n$, both $S_{\leq_\scrJ i}$ and $S_{\leq i}$ are ideals
of $\catC$. 
\end{proposition}

\begin{proof}
Let $i\in\{1,\ldots,n\}$, let $s\in S_i$, and let $u,v\in S$ be such
that $s\circ u$ and $v\circ s$ exist. Since 
$S\circ s\circ u\circ S\subseteq S\circ s\circ S$ and 
$S\circ v\circ s\circ S\subseteq S\circ s\circ S$, we immediately get
$\scrJ(s\circ u)\leq_\scrJ \scrJ(s)$ and $\scrJ(v\circ s)\leq_\scrJ \scrJ(s)$.
Thus $S_{\leq_{\scrJ}i}$ is an ideal of $\catC$, for $i=1,\ldots,n$, and
since $S_{\leq i}=\bigcup_{j\leq i}S_{\leq_\scrJ j}$, the latter is an ideal 
of $\catC$ as well.
\end{proof}

\begin{proposition}\label{prop right ideals}
{\rm (a)}\, Let $s,t\in S$ be such that $s=s\circ t\circ s$. 
Then $S\circ s=S\circ t\circ s$.

\smallskip
{\rm (b)}\, Let $i\in\{1,\ldots,n\}$, and let
$s,t\in S_i$ be such that $S\circ s\subseteq S\circ t$. 
Then $S\circ s=S\circ t$.

\smallskip
{\rm (c)}\, Let $i\in\{1,\ldots,n\}$, and let $s,t\in S_i$. Then
the sets $(S\circ s)_i:=(S\circ s)\cap S_i$ 
and $(S\circ t)_i:=(S\circ t)\cap S_i$ are either
equal of disjoint.

\smallskip
{\rm (d)}\, Let $i\in\{1,\ldots,n\}$.
There is a subset $\epsilon$ of $[e_i]$ such that the 
sets $(S\circ e)_i:=(S\circ e)\cap S_i$ ($e\in \epsilon$)
form a partition of $S_i$.
\end{proposition}

\begin{proof}
(a) This follows from 
$S\circ s=S\circ s\circ t \circ s\subseteq S\circ t\circ s \subseteq S\circ s$.

\smallskip
(b) Let $q,r\in S$ be such that $s\circ q \circ s=s$ and
$t\circ r\circ t=t$, and 
set $e:=q\circ s$ and $f:=r\circ t$. 
Since $s\in S\circ t$, the idempotents $e$ and $f$ are 
endomorphisms of the same object of $\catC$, say $X$. By Part~(a), we 
have $S\circ e\subseteq S\circ f$, and it suffices to show that 
$S\circ f\subseteq S\circ e$. Recall that $\scrJ(s)=S_i=\scrJ(t)$.
Thus, by Lemma~\ref{lemma id equiv}, we have $e\sim f$, so that
there exist $u\in e\circ S\circ f$ and $v\in f\circ S\circ e$ 
with $e=u\circ v$ and $f=v\circ u$. Since $S\circ e\subseteq S\circ f$, we 
also have $e=e\circ f$. Note that  $u$ and $v$ are endomorphisms of $X$. 
Since $\End_{\catC}(X)$ is finite,
there exist positive integers $a$ and $b$ such that 
$v^{a+b}=v^b$. Composition with $u^b$ from the right yields $v^a=f$, since 
we have $v\circ u=f$ and $v\circ f=v\circ e\circ f=v\circ e=v$. Finally, we 
obtain $f\circ e=v^a\circ e=v^a=f$, which implies that $f\in S\circ e$ and 
$S\circ f\subseteq S\circ e$. 

\smallskip
(c) Assume that $(S\circ s)_i\cap (S\circ t)_i$ is non-empty and that 
$u\in(S\circ s)_i\cap(S\circ t)_i$. 
Then we obtain $S\circ u\subseteq S\circ s$ and $S\circ u\subseteq S\circ t$. 
Now Part~(b) yields $S\circ s=S\circ u=S\circ t$ and 
$(S\circ s)_i=(S\circ t)_i$.

\smallskip
(d) Clearly, $S_i$ is the union of its subsets $(S\circ s)_i$, $s\in S_i$, 
and by Part~(a) also of the subsets $(S\circ e)_i$, $e\in [e_i]$. 
The condition $(S\circ e)_i = (S\circ f)_i$ defines an equivalence relation on 
the set $[e_i]$. 
If $\epsilon$ is a set of representatives of the corresponding
equivalence classes
then, by Part~(c), $S_i$ is the disjoint union of the subsets 
$(S\circ e)_i$, $e\in\epsilon$.
\end{proof}

\begin{theorem}\label{thm qh}
Let $\catC$ be a finite split category and let $\alpha$ be a $2$-cocycle of $\catC$ with values in
the multiplicative group $k^\times$ of a field $k$. Assume further that, for each idempotent $e$ of $\catC$, 
the order of $\Gamma_e$ is invertible in $k$.
With the notation as in Definition~\ref{defi ideals S}, let 
$J_i:=kS_{\leq i}$, for $i=0,\ldots,n$. Then
\begin{equation}\label{eqn qh chain}
\{0\}=J_0\subset J_1\subset\cdots \subset J_n=k_\alpha\catC
\end{equation}
is a heredity chain for $k_\alpha\catC$. In particular, the twisted category algebra $k_\alpha\catC$ is quasi-hereditary.
\end{theorem}

\begin{proof}
We set $A:=k_\alpha\catC$. Since $S_{\le i}$ is an ideal of $S$, $J_i$ is an ideal of $A$ for all $i=0,\ldots,n$.
We show that the chain (\ref{eqn qh chain})
satisfies conditions (i), (ii) and (iii') in Definition~\ref{defi qh}. To this end,
let $i\in\{1,\ldots,n\}$, and again let $\bar{\cdot}\colon A\to A/J_{i-1}$
denote the canonical epimorphism. 

\smallskip
By definition, we have $S\circ s\circ S= S\circ e_i\circ S$, for 
every $s\in S_i$, thus $S_i\subseteq S\circ e_i\circ S$.
From this we get $\bar{J}_i=\bar{A}\bar{e}_i\bar{A}$, and we have verified
condition~(i).

\smallskip 
Next we verify condition~(ii). 
Note that, since $A$ is a finite-dimensional
algebra over a field, we have $\overline{\mathbf{J}(A)}=\mathbf{J}(\bar{A})$.
Hence it suffices to show that
$sut\in J_{i-1}$, for all
$s,t\in S_{\leq i}$ and all
$u\in\mathbf{J}(A)$. If $s\in S_{\leq i-1}$ or $t\in S_{\leq i-1}$ then this
is clearly true. So we may suppose that $s,t\in S_i$. Let $q,r\in S$ be such that
$s\circ q\circ s=s$ and $t\circ r\circ t=t$. Then 
$S\circ s\circ S=S\circ e_i\circ S=S\circ t\circ S$ and 
$[e_i]=[s\circ q]=[q\circ s]=[t\circ r]=[r\circ t]$, by
Lemma~\ref{lemma id equiv}.
So there exist elements $x\in q\circ s\circ S\circ e_i$, 
$y\in e_i\circ S\circ q\circ s$,
$v\in t\circ r\circ S\circ e_i$, and $w\in e_i\circ S\circ t\circ r$ such that
$$q\circ s=x\circ y=x\circ e_i\circ y,\quad t\circ r=v\circ w=v\circ e_i\circ w,\quad e_i=y\circ x=w\circ v\,.$$
Since $u\in\mathbf{J}(A)$, Proposition~\ref{prop radical}
implies $(e_i\circ y)u(v\circ e_i)\in e'_iAuAe'_i\subseteq kJ_{e_i}$.
Furthermore, we have $e_i\circ S\circ e_i\subseteq S_{\leq i}$, since 
$e_i\in S_i\subseteq S_{\leq i}$ and
since, by Proposition~\ref{prop ideals S}, 
$S_{\leq i}$ is an ideal in $S$. By (\ref{eqn Gamma J}) we have 
$kJ_{e_i}\subseteq J_{i-1}$. This implies $(e_i\circ y)u(v\circ e_i)\in J_{i-1}$ and
we obtain
$$su t=(s\circ q\circ s)u(t\circ r\circ t)=(s\circ x\circ e_i\circ y)u(v\circ e_i\circ w\circ t)\in
J_{i-1}\,,$$
as required.

\smallskip
It remains to verify condition~(iii'). By 
Proposition~\ref{prop right ideals}(d), we know that
$\bar{J}_i$ is the direct sum of left ideals of the form
$\bar{A}\bar{e}$, for suitable idempotents $e\in A$. Since 
each such summand is a projective left $\bar{A}$-module, so is $\bar{J}_i$,
and the proof of (iii') is complete.
\end{proof}


\section{Standard modules for $k_\alpha\catC$}\label{sec std mod}

As before, we denote the twisted category algebra $k_\alpha\catC$ by $A$. As in Theorem~\ref{thm qh} we assume throughout this section that $|\Gamma_e|\in k^\times$ for each idempotent $e$ of $\catC$.

\begin{nothing} {\bf The modules $\Delta_{ir}$.}\label{noth Delta}\, Recall from \ref{not simples} that the heads of the $A$-modules
\begin{equation*}
  \Delta_{ir}:=Ae_i'\otimes_{e_i'Ae_i'}\tilde{T}_{ir}\quad (i\in\{1,\ldots,n\},r\in\{1,\ldots,l_i\}) 
\end{equation*}
yield a set of representatives of the isomorphism classes of simple $A$-modules.
Here, $T_{i1},\ldots,T_{il_i}$ again denote representatives
of the isomorphism classes of simple $k_\alpha\Gamma_{e_i}$-modules.

From now on we set $\Lambda:=\{(i,r)\mid 1\leq i\leq n,\, 1\leq r\leq l_i\}$,
and we define a partial order $\leq$ on $\Lambda$ via
\begin{equation}\label{eqn lambda poset}
(i,r)<(j,s):\Leftrightarrow S_j<_\scrJ S_i\,.
\end{equation}
\end{nothing}

The aim of this section is to prove the following theorem.

\begin{theorem}\label{thm std modules}
The modules $\Delta_{ir}$ ($(i,r)\in\Lambda$) are the standard modules
of the quasi-hereditary algebra $A$ with respect to the partial order $\le$ on $\Lambda$.
\end{theorem}

\begin{remark}\label{rem strategy}
(a) In order to prove Theorem~\ref{thm std modules}, we will 
have to show that,
for each $(i,r)\in\Lambda$, the
$A$-module $\Delta_{ir}$ 
is the unique maximal quotient module  $M$ of the projective cover $P_{ir}$ of $D_{ir}$
such that all composition factors of $\mathrm{Rad}(M)$ belong to
the set $\{D_{js}\mid (j,s)<(i,r)\}$ (see the definition of a standard module with respect to the partial order $\le$ in \cite[A1]{D}).

Note that it thus suffices to show that, for each $(i,r)\in\Lambda$, the
module $\Delta_{ir}$
satisfies conditions
(i) and (ii) below, and the projective
cover $P_{ir}$ of $D_{ir}$ admits a filtration
\begin{equation}\label{eqn proj filtration}
0=P_{ir}^{(0)}\subset \cdots\subset P_{ir}^{(m_{ir})}=P_{ir}
\end{equation}
satisfying (iii) and (iv) below:

\smallskip
(i)\, $\Hd(\Delta_{ir})\cong D_{ir}$;

\smallskip
(ii)\, if $(j,s)\in\Lambda$ is such that $D_{js}$ occurs as a
composition factor of $\mathrm{Rad}(\Delta_{ir})$ then $(j,s)<(i,r)$;

\smallskip
(iii)\, $P_{ir}^{(m_{ir})}/P_{ir}^{(m_{ir}-1)}\cong \Delta_{ir}$;

\smallskip
(iv)\, for $q\in\{1,\ldots,m_{ir}-1\}$, one has $P_{ir}^{(q)}/P_{ir}^{(q-1)}\cong \Delta_{j_q,s_q}$, for some $(j_q,s_q)\in\Lambda$ with $(i,r)<(j_q,s_q)$.

\smallskip
(b) Note also that the partial order on $\Lambda$ defined in \cite[Proposition~A3.7(ii)]{D} from the heredity chain~(\ref{eqn qh chain}) in Theorem~\ref{thm qh} is finer than the partial order defined in (\ref{eqn lambda poset}). Thus, the modules $\Delta_{ir}$ are also the standard modules associated with the heredity chain in~(\ref{eqn qh chain}).
\end{remark}

We start out with the following lemma, which will be essential for
verifying conditions (i)--(iv) above. As in Theorem~\ref{thm qh}, given
$i\in\{1,\ldots n\}$, we
denote by $J_i$ the ideal $kS_{\leq i}$ of $A$, and we also set $J_0:=\{0\}$.

\begin{lemma}\label{lemma ideal quot}
Let $i\in\{1,\ldots,n\}$, and let $\epsilon_i\subseteq[e_i]\subseteq S_i$
be a set of idempotents such that the sets $(S\circ e)_i:=(S\circ e)\cap S_i$
form a partition of $S_i$ (cf.~Proposition~\ref{prop right ideals}).
Then one has a left $A$-module isomorphism
\begin{equation}\label{eqn nir}
J_i/J_{i-1}\cong \bigoplus_{r=1}^{l_i}\Delta_{ir}^{|\epsilon_i|\cdot n_{ir}}\quad \text{ with }\quad n_{ir}:=\frac{\dim_k(T_{ir})}{\dim_k(\End_{k_\alpha\Gamma_{e_i}}(T_{ir}))}\,.
\end{equation}
\end{lemma}

\begin{proof}
We first show that, for each $e\in\epsilon_i$, we have
\begin{equation}\label{eqn intersection}
  S_{\leq i-1}\cap S\circ e=S\circ J_e\,.
\end{equation}
By (\ref{eqn Gamma J}), we know that $J_e\subseteq S_{\leq i-1}$, so that 
$S\circ J_e$ is contained in $S_{\leq i-1}\cap S\circ e$. Conversely, if
$x\in S_{\leq i-1}\cap S\circ e$ then $x=x\circ e$, and there is some
$y\in S$ such that $x=x\circ y\circ x$. Thus $x=x\circ e\circ y\circ x\circ e$,
and $e\circ y\circ x\circ e\in e\circ S\circ e\cap S_{\leq i-1}=J_e$, by (\ref{eqn Gamma J}).
Thus $x\in S\circ J_e$, as claimed.

\smallskip
Equation~(\ref{eqn intersection}) implies $J_{i-1}\cap Ae' = AJ_e$ and we obtain the following left $A$-module isomorphisms
\begin{align*}
  J_i/J_{i-1}&=k\bigl(S_{\leq i-1}\,\disj\,\bigdisj_{e\in \epsilon_i}(S\circ e)_i\bigr)/J_{i-1}
  =\bigoplus_{e\in\epsilon_i}    (J_{i-1}+Ae')/J_{i-1}\\
  & \cong \bigoplus_{e\in\epsilon_i} Ae'/(Ae'\cap J_{i-1})=\bigoplus_{e\in\epsilon_i} Ae'/AJ_e\cong 
  \bigoplus_{e\in\epsilon_i} Ae'\otimes_{e'Ae'}(e'Ae'/kJ_e)\,,
\end{align*}
where the last isomorphism is given by the canonical isomorphisms
$$Ae'\otimes_{e'Ae'}(e'Ae'/kJ_e)=Ae'\otimes_{e'Ae'}(e'Ae'/e'Ae'J_e)\cong Ae'/Ae'J_e=Ae'/AJ_e\,.$$
By (\ref{eqn indep of e}), $Ae'\otimes_{e'Ae'}(e'Ae'/kJ_e)\cong Ae_i'\otimes_{e_i'Ae_i'}(e_i'Ae_i'/kJ_{e_i})$, for every $e\in\epsilon_i$. Moreover, recall that the twisted group algebra
$k_\alpha\Gamma_{e_i}$ is semisimple, since we are assuming 
$|\Gamma_{e_i}|\in k^\times$.
Hence, since $T_{i1},\ldots, T_{il_i}$ are representatives
of the isomorphism classes of simple $k_\alpha\Gamma_{e_i}$-modules, $k_\alpha\Gamma_{e_i}\cong \bigoplus_{r=1}^{l_i}T_{ir}^{n_{ir}}$ as
left $k_\alpha\Gamma_{e_i}$-modules, where the multiplicity $n_{ir}$
is as in (\ref{eqn nir}).
Then also
$$Ae_i'\otimes_{e_i'Ae_i'}(e_i'Ae_i'/kJ_{e_i})\cong Ae_i'\otimes_{e_i'Ae_i'}\bigoplus_{r=1}^{l_i}\tilde{T}_{ir}^{n_{ir}}\cong \bigoplus_{r=1}^{l_i}(Ae_i'\otimes_{e_i'Ae_i'}\tilde{T}_{ir})^{n_{ir}}\cong \bigoplus_{r=1}^{l_i}\Delta_{ir}^{n_{ir}}$$
as left $A$-modules. Altogether this gives the desired left $A$-module isomorphism
$$J_i/J_{i-1}\cong \bigoplus_{r=1}^{l_i}\Delta_{ir}^{n_{ir}\cdot |\epsilon_i|}\,.$$
\end{proof}

\begin{proof}{\bf(of Theorem~\ref{thm std modules})}
We follow the strategy outlined in Remark~\ref{rem strategy} and verify
conditions~(i)--(iv) listed there.

\smallskip
(i)\, This follows immediately from Green's condensation theory
and the definition of the modules $D_{ir}$, see~\ref{noth simples}(b)
and \ref{not simples}.

\smallskip

(ii)\, Suppose that $(i,r),(j,s)\in\Lambda$ are such that 
$D_{js}$ occurs as a composition factor of $\Rad(\Delta_{ir})$.
We need to show that $(j,s)<(i,r)$, i.e., that $S_i<_{\scrJ} S_j$.
By \cite[Proposition~5.1]{LS}, one has $S_j\cdot D_{js}\neq \{0\}$, and
if $l\in\{1,\ldots,n\}$ is such that $S_l\cdot D_{js}\neq \{0\}$ then
$S_j\leq_\scrJ S_l$. Assume that $S_i\not<_\scrJ S_j$. Then $S_j\circ S_{\leq i}\subseteq S_{\leq i-1}$,
which, by Lemma~\ref{lemma ideal quot}, implies 
$S_j\cdot \Delta_{ir}=\{0\}$. Thus, $S_j$ also annihilates every composition factor
of $\Delta_{ir}$ and we have $S_j\cdot D_{js}=\{0\}$, a contradiction.

\smallskip
(iii) and (iv):\, By Lemma~\ref{lemma ideal quot}, the left $A$-module $A$
has a filtration all of whose factors are isomorphic to modules
of the form $\Delta_{ir}$ ($(i,r)\in\Lambda$). Let
$(i,r)\in\Lambda$, and let $f_{ir}\in e_i'Ae_i'$ be a primitive
idempotent such that $e_i'Ae_i'f_{ir}=e_i'Af_{ir}$ is a projective
cover of the simple $e_i'Ae_i'$-module $\tilde{T}_{ir}$.

We claim that $P_{ir}:=Af_{ir}$ is a projective cover of the $A$-module
$D_{ir}$. Since $f_{ir}$ is primitive in $e_i'Ae_i'$, it is primitive
in $A$ as well, so that $P_{ir}$ is an indecomposable projective $A$-module.
We have an $e_i'Ae_i'$-epimorphism
$e_i'Af_{ir}\twoheadrightarrow \tilde{T}_{ir}$, and thus also an $A$-epimorphism
$$Af_{ir}\cong Ae_i'\otimes_{e_i'Ae_i'}(e_i'Ae_i'f_{ir})\twoheadrightarrow Ae_i'\otimes_{e_i'Ae_i'}\tilde{T}_{ir}=\Delta_{ir}\twoheadrightarrow D_{ir}\,,$$
since tensoring with $Ae_i'$ is right exact.
Hence $P_{ir}$ must be a projective cover of $D_{ir}$.

Since $Af_{ir}=Ae_i'f_{ir}\subseteq kS_{\leq i}f_{ir}=J_if_{ir}\subseteq Af_{ir}$,
the filtration of the left $A$-module $A$ in (\ref{eqn qh chain})
yields a filtration
\begin{equation}\label{eqn Pir}
  \{0\}=J_0\cdot f_{ir}\subseteq J_1\cdot f_{ir}\subseteq\cdots\subseteq J_{i-1}\cdot f_{ir}\subseteq 
  J_i\cdot f_{ir}=Af_{ir}\,.
\end{equation}
If $j\in\{1,\ldots,i-1\}$ is such that
$S_j\not<_\scrJ S_i$ then we have
$$J_jf_{ir}\subseteq J_j\cdot e_i'Ae_i'\subseteq kS_{\leq j}\cdot kS_{\leq_\scrJ i}=k(S_{\leq j}\cdot S_{\leq_\scrJ i})\subseteq k(S_{\leq j}\cap S_{\leq_\scrJ i})\subseteq J_{j-1}\,,$$
since $S_j\cap S_{\leq\scrJ i}=\emptyset$. So in this case
we get $J_jf_{ir}=J_{j-1}f_{ir}$.

If $j\in\{1,\ldots,i-1\}$ is such that
$S_j<_\scrJ S_i$ then, by Lemma~\ref{lemma ideal quot}, we deduce that, 
for $j=1,\ldots,i-1$,
the indecomposable direct summands of the
quotient $J_j\cdot f_{ir}/J_{j-1}\cdot f_{ir}$ are again isomorphic to
$\Delta_{js}$, for various $s\in\{1,\ldots,l_j\}$ such that $(i,r)<(j,s)$.

\smallskip

It thus suffices to show that $Af_{ir}/J_{i-1}f_{ir}\cong \Delta_{ir}$. 
For then the chain (\ref{eqn Pir}) gives rise to
a filtration of $P_{ir}$ as desired.

Since $e_i'Ae_i'f_{ir}$ is a projective cover of $\tilde{T}_{ir}$,
we have $f_{ir}\cdot \tilde{T}_{ir}\neq \{0\}$. Since the ideal $kJ_{e_i}$ in $e_i'Ae_i'$
annihilates $\tilde{T}_{ir}$ and since $kJ_{e_i}=e_i'Ae_i'\cap J_{i-1}$,
this implies $f_{ir}\notin J_{i-1}$, hence 
$f_{ir}\in e_i'Ae_i'\smallsetminus J_{i-1}\subseteq J_i\smallsetminus J_{i-1}$.
Therefore, $Af_{ir}/J_{i-1}f_{ir}\neq \{0\}$.
Since $Af_{ir}$ has a simple head isomorphic to $D_{ir}$, the same holds for $Af_{ir}/J_{i-1}f_{ir}$.
Finally, since also $\Delta_{ir}$ has a simple head isomorphic
to $D_{ir}$, Lemma~\ref{lemma ideal quot} forces $Af_{ir}/J_{i-1}f_{ir}\cong \Delta_{ir}$,
and the proof of the theorem is complete.
\end{proof}

\end{document}